\documentclass[11pt]{article}

\usepackage[utf8]{inputenc}

\usepackage{graphicx,subfigure}
\usepackage{color}
\usepackage{algorithm}
\usepackage{algpseudocode}
\usepackage{amsmath,amsfonts,amscd,amssymb,amsthm}
\usepackage{bm,bbm}
\usepackage{algorithm,algpseudocode}
\usepackage[plainpages=false]{hyperref}
\usepackage{fancyhdr}
\usepackage{verbatim}
\usepackage{array}
\usepackage{fullpage}
\usepackage{enumitem}

\usepackage{titlesec}
\titleformat{\section}
	{\normalfont\Large\bfseries\filcenter}{\thesection.}{1 ex}{}
\titleformat{\subsection}%[runin]
	{\normalfont\normalsize\bfseries}{\thesubsection.}{1 ex}{}
\titleformat{\subsubsection}[runin]
	{\normalfont\normalsize\bfseries\filcenter}{\thesubsubsection.}{1 ex}{}

\usepackage[charter]{mathdesign}
\usepackage[mathcal]{eucal}

\usepackage[margin=1.2 in]{geometry}

\usepackage[square,numbers,sort&compress]{natbib}

\usepackage{thmtools,thm-restate}
\declaretheoremstyle[headpunct={ --- },headfont=\normalfont\itshape]{myremark}
\declaretheoremstyle[bodyfont=\normalfont]{mydefinition}

\declaretheorem[name=Theorem,within=section]{Thm}
\declaretheorem[within=section,name=Lemma]{Lem}
\declaretheorem[sibling=Lem,name=Proposition]{Prop}

\declaretheorem[style=myremark,sibling=Lem,name=Remark]{Rem}
\declaretheorem[style=mydefinition,sibling=Lem,name=Definition]{Def}

\declaretheorem[style=mydefinition,sibling=Lem,name=Example]{Ex}

\declaretheorem[style=mydefinition,sibling=Lem,name=Problem]{Prob}

\newcommand{\iprod}[2]{\left\langle {#1},{#2}\right\rangle}

\newcommand{\R}{\ensuremath{\mathbb{R}}}		% real numbers
\newcommand{\C}{\ensuremath{\mathbb{C}}}		% complex numbers
\newcommand{\A}{\ensuremath{\mathbb{A}}}
\newcommand{\E}{\ensuremath{\mathbb{E}}}
\newcommand{\K}{\ensuremath{\mathbb{K}}}

\newcommand{\supp}{\operatorname{supp}}

\newcommand{\Ker}{\operatorname{Ker}}

\newcommand{\Tr}{\operatorname{Tr}}

\newcommand{\Var}{\operatorname{Var}}

\newcommand{\spn}{\operatorname{span}}
\newcommand{\aff}{\operatorname{aff}}

\newcommand{\calC}{\mathcal{C}}

\newcommand{\calI}{\mathcal{I}}

\newcommand{\calL}{\mathcal{L}}

\newcommand{\mA}{\mathbf{A}}

\newcommand{\vu}{\mathbf{u}}
\newcommand{\vv}{\mathbf{v}}

\makeatletter
\providecommand*{\diff}%
        {\@ifnextchar^{\DIfF}{\DIfF^{}}}
\def\DIfF^#1{%
        \mathop{\mathrm{\mathstrut d}}%
                \nolimits^{#1}\gobblespace
}
\def\gobblespace{%
        \futurelet\diffarg\opspace}
\def\opspace{%
        \let\DiffSpace\!%
        \ifx\diffarg(%
                \let\DiffSpace\relax
        \else
                \ifx\diffarg\[%
                        \let\DiffSpace\relax
                \else
                        \ifx\diffarg\{%
                                \let\DiffSpace\relax
                        \fi\fi\fi\DiffSpace}
\makeatother

\begin{document}

\title{Matroid Regression}

\author{Franz J. Király \thanks{Department of Statistical Science, Univerity College London, and MFO
	 \url{f.kiraly@ucl.ac.uk}} \and
	 Louis Theran\thanks{Inst. Math., AG Diskrete Geometrie, Freie Universität Berlin, \url{theran@math.fu-berlin.de}}}

\date{}

\maketitle

\begin{abstract}
\begin{normalsize}
We propose an algebraic combinatorial method for solving large sparse
linear systems of equations locally - that is, a method which can compute single evaluations of
the signal without computing the whole signal. The method scales only in the sparsity of the
system and not in its size, and allows to provide error estimates for any solution method.
At the heart of our approach is the so-called regression matroid, a combinatorial object associated
to sparsity patterns, which allows to replace inversion of the large matrix with the inversion of a
kernel matrix that is constant size. We show that our method provides the best linear unbiased estimator (BLUE) 
for this setting and the minimum variance unbiased estimator (MVUE) under Gaussian noise assumptions, 
and furthermore we show that the size of the kernel matrix which is to be inverted can be traded off with accuracy.
\end{normalsize}
\end{abstract}

\section{Introduction}
Sparse linear systems are a recurring topic in modern science. They occur in a wide variety of contexts such as numerical analysis, medical imaging, control theory or signal processing. Of particular interest are sparse linear systems of big size - that is, a number of equations which is of the order of thousands, millions, billions or more - since they occur in practice, e.g. in linear inverse problems such as tomography, or analysis of large scale data with sparse structure as they occur in recommender systems or network analysis.\\

Whole areas of research, spanning disciplines in most areas of science, have been devoted to the end of solving linear systems of equations $Ax = b$ where $A$ is huge and sparse. A selection of books on the topic, in which numerical solution strategies are outlined, and which is far from being representative, includes~\cite{Tewarson73,Hackbusch1993,Barrett1994,Saad2003,Davis06}. Further important is the area in medical imaging which is concerned with sparse linear systems specifically arising from certain geometries in tomography, compare the algebraic approaches in ~\cite{Herman1980,Kak1988}, for which specific techniques have been developed. Moreover, we would like to mention that sparse matrices and their spectral properties also appear as a recurring topic in networks, see e.g.~\cite{Chung1997}.

Regarding the huge corpus of existing literature, we would, however, like to stress one fact: the state-of-the-art methods and theories mostly make use of spectral or analytical properties of the huge matrices; efficient methods which use particular structure - be it algebraic or combinatorial - of the sparse system of equations, seem not to be available. Furthermore, all methods usually seek a complete solution of the system in terms of $x$, while scenarios where some projection $P x$ with $P$ a matrix and $P x$ of tractable size might be desirable - e.g., if in the tomography scenario, only part of the scanned region needs a high resolution, or in the networks scenario, where only part of the network might be of, say, predictive interest. Similarly, in the recommender systems scenario, it is more natural to make a recommendation for a single item instead of making all possible recommendations at once.\\

In this paper, we propose theoretical foundations and practical methods to address this kind of problem, which have the potential advantage of scaling with the row-size of $P$ instead of the size of $A$. That is, optimally the method will have a running time that does not scale with the size of $A$, only with certain sparsity properties of $A$ which in many practical scenarios scale constant with respect to the size of $A$. The only assumption we will need for this to work is that there exist a sufficient number of linear dependencies of rows of $A$ which are sparse in their coefficient representation. This is frequently the case if $A$ has intrinsic combinatorial meaning, or is highly structured otherwise.\\

The central ingredient is the notion of regression matroid, which provides a kind of dictionary for minimal such dependencies (= circuits), and the circuit kernel matrix, which is the covariance matrix between the circuits. Restricting to small circuits in a "neighborhood" of $P$, we are able to obtain a least squares estimator for $P x$ where the most costly ingredient is inversion of the circuit kernel matrix - which scales with the size and number the circuits, and not the size of $A$. Therefore, through choosing the circuits - interpreted as the ``locality'' parameter of matroid regression - we also obtain a tool of trading off accuracy of the solution with computational cost.

More concisely, our main contributions are:
\begin{itemize}
\item the notions of {\bf matroid regression} and {\bf circuit kernel}, capturing algebraic combinatorial properties of the linear system
\item an {\bf explicit algorithm} computing a variance minimizing estimator for the evaluation $P x$
\item an explicit form for the {\bf variance of that estimator} which depends not on $x$ but only on the noise model
\item an {\bf explicit algorithm} to compute that variance without computing $P x$
\item {\bf proofs of optimality} and universality for the estimator (BLUE in general, MVUE for Gaussian noise and for unknown noise)
\item a proof of the error being monotonous in the ``locality'' of the estimate, yielding a {\bf complexity-accuracy-tradeoff}
\item {\bf characterization of the regression matroid} in some cases, including potential measurements, 2-sparse vectors, rank one matrix completion; explanation how in these cases circuits and combinatorial properties of {\bf characteristic graphs} relate
\end{itemize}
Our framework also explains some particular findings in the case of matrix completion~\cite{KirThe13Rank1}, which we can reproduce by reduction to a sparse linear system, and solves open questions about the optimality of the estimators raised in~\cite{KirThe13Rank1}. In the same sense, we hypothesize that the matroid regression methods have a rather general and natural extension to the non-linear case.

\section{Structured linear estimation}\label{sec:theory}
We will consider two compressed sensing problems which are dual to each other.  In the sequel, the
field $\K$ is always one of $\C$ or $\R$, and the parameter $n$ will be the signal dimension.
\begin{Prob}[name={Primal problem \textbf{(P)}}]\label{Prob:primal}
There is an unknown signal $x\in \K^n$ observed via a
\emph{linear measurement process}:
\begin{equation}\label{Eq:primal}
	b = Ax + \varepsilon
\end{equation}
The noise $\varepsilon$ is centered and has finite variance, and the matrix $A\in\K^{N\times n }$ is known.
The task is to compute a \emph{linear evaluation} $\gamma = \iprod{w}{x}$, for a known $w\neq 0$ in the
row-span $\spn A$.
\end{Prob}
In general, $n$ will be large, and potentially $N\gg n$, but $A$ will be either \emph{sparse},
\emph{structured} or both.  This means that simply inverting an $n\times n$
sub-matrix of $A$ is not a good solution.  Instead, we will show how to use the structure
of $A$ to find solutions \emph{locally}, using very few coordinates of $b$ or both.

\begin{Prob}[name={Dual problem \textbf{(D)}}]
There is an unknown signal $y\in \K^n$ and an unknown scalar $\gamma$, satisfying a constraint
\begin{equation}
	\gamma w' = A' y
\end{equation}
The task is to estimate $\gamma$ from observations $b = y + \varepsilon$, with $A'\in \K^{N\times n}$
and $w'$ known.
\end{Prob}
Since the Dual Problem (D) can be treated with the same methods, we will
focus on the Primal problem (P).

\subsection{The Problems in Context}
We interpret the general problem (P) as a \emph{supervised learning} problem.
To see this, take the rows $a_1,a_2,\ldots, a_N$ to be training data
points and the coordinates $b_1,\ldots, b_N$ of $b$ to be training labels.
The unknown vector $x$ is then the regressor, and the learning tasks
can be: (i) imputation of single coordinates of $x$; (ii) prediction
of the label of a new point $w$; (iii) denoising, which corresponds to
$w$ being one of the $a_i$; among others.

Alternatively, even though $N$ is typically quite large, so that $x$ is
not compressed in the classic sense, problem (P) can be interpreted
in terms of \emph{compressed sensing}.  Here, the task is to use as
few coordinates of $b$ as possible to estimate $\iprod{w}{x}$ accurately.
This should be contrasted with the approach of computing the (pseudo-)inverse
of $A$, e.g., for i.i.d.~noise the estimator $w^\top A^{-1}b$.

\subsection{Example Instances}
To fix, the concept, we show how to cast some scenarios in terms of problem (P).

\begin{Ex}[name={Measuring potentials}]\label{Ex:potentials}
The task is to measure from an unknown potential $x$, given a set of
measurements.  The rows of $A$ are of the form $e_j - e_i$, where
$\{e_i : i\in [n]\}$ are the standard basis vectors of $\K^n$.  The
vector $w$ is also of this form.
\end{Ex}

\begin{Ex}[name={Rank 1 Matrix Completion}]\label{Ex:matrix-completion}
The task is to impute or denoise the entry at position $(i,j)$ in a partially-observed, $m\times n$
rank $1$ matrix $\mA = \vu\cdot\vv^{\top}$.  The vector $x$ is the concatenation of the
entry-wise logarithms of $\vu$ and $\vv$; the vector $e_{j + m} - e_i$ is a row of $A$
if the position $(i,j)$ is observed; the vector $b$ is the vectorization of the set of
observed entries; $w = e_j - e_i$ where $(i,j)$ is the position of the entry to impute
or de-noise.
\end{Ex}

\begin{Ex}[name={Discrete tomography}]\label{Ex:tomography}
The task is to reconstruct a bitmap image (subset of a lattice in Euclidean space) from
a number of projections.  The matrix $A$ has a decaying spectrum and $w$ describes a regularized region of interest.
\end{Ex}

\section{Regression matroids}
Our strategy for solving the problem (P) will be to exploit the \emph{structure} of the constraint matrix
$A$.  The object that captures this is the \emph{regression matroid} of $A$ and $w$, which we now define.
\begin{Def}
Let $a_1,\dots, a_N\in\K^n$ be a collection of vectors, let $w\in\K^n$ be a target vector.
\begin{description}
\item{(i)} The (linear) \emph{regression matroid} associated to the $a_i$ and $w$ is the pair $([N],\calI)$, where
\[
	\calI := \left\{ I\subseteq [N] : \text{the set $\{w\}\cup \{a_i : i\in I\}$ is linearly independent}\right\}
\]
and we write $a_\ast := w$. We will denote the matroid by $L (w|a_1,\dots, a_N):=([N],\calI)$. If $A$ is the matrix having
$a_i$ as $i$-th row, we simply write $L (w|A)$.
\item{(ii)} A set $C\subseteq [N]$ with $\ast\in C$ is called (linear) \emph{particular regression circuit} of $L(w|A)$, if the equation
$ w = \sum_{i\in C} \lambda_i a_i$
implies $\lambda_i\neq 0$ for all $i\in C$.
\item{(iii)} A set $C\subseteq [N]$ is called \emph{general regression circuit} of $L(w|A)$, if it is a regression circuit of $L(0|A)$.
\end{description}
\end{Def}
In matroid terms, the regression matroid is the elementary quotient of the linear matroid of $A$ by
the element $w$. Note that a set $C\subseteq [N]$ can not be both a particular and general regression circuit. Also, if $w$ is one of the $a_i$, then $\{i\}$  is a particular regression circuit. An extension to more than one target vector is straightforward, but for simplicity, we continue with only the single target vector $w$.

If $\lambda\in \K^n$ is a vector, we say that the \emph{support} of $\lambda$ is the
set $\{i\in [N]: \lambda_i\neq 0\}$.  Circuits and regression circuits correspond to
linear dependencies with minimal support.

\begin{Prop}\label{Prop:circdual}
Let $L(w|A)$ be a regression matroid. Then:
\begin{itemize}
	\item[(i)] $C\subseteq [N]$ is a particular regression circuit if and only if there is a
		unique vector $\lambda\in \K^N$ supported on $C$ such that $w = \lambda A$.
	\item[(ii)] $C\subseteq [N]$ is a general regression circuit if and only if
		there is a unique, up to scalar multiplication, vector $\lambda\in \K^n$ supported on $C$
		such that $\lambda A = 0$.
\end{itemize}
\end{Prop}
\begin{proof}
We will prove (i), since the proof of (ii) is similar.  Suppose that
$\lambda_1 A = w$ and $\lambda_2 A = w$ on a set $C$, and let $j\in [N]$ be arbitrary.
Set $\alpha = \lambda_1(j)/\lambda_2(j)$.  Then $(\alpha\lambda_2 - \lambda_1)A = (\alpha-1)w$,
and the support of $(\alpha\lambda_2 - \lambda_1)\subsetneq C$.  In particular, $C$ has minimal
support if and only if $\alpha = 1$.  Since $j$ was arbitrary, we are done.
\end{proof}
Proposition~\ref{Prop:circdual} justifies the following definition:
\begin{Def}
Let $w\in\K^n$, let $A\in\K^{N\times n}$, let $C$ be a particular regression circuit of $L(w|A)$.
We call the unique vector $\lambda$ associated to $C$ by
Proposition~\ref{Prop:circdual} the \emph{circuit vector}
of $C$. Similarly, we may pick a normed
representative to define the \emph{circuit vector} of a general circuit.
For a circuit $C$ of either type (recall that a circuit can be particular or general, but not both), we will use $\lambda_C$ to denote its circuit vector.
\end{Def}

We introduce a definition for formal linear combinations of circuits:

\begin{Def}
For circuits $C_1,\dots, C_m$ and $\alpha_1,\dots, \alpha_m\in \K$, we will define a \emph{circuit divisor} to be a formal linear combination
$$\alpha C_1+\dots + \alpha_m C_m,$$
and associate to it the circuit vector $\alpha_1 \lambda_{C_1}+\dots+\alpha_m \lambda_{C_m}$. We denote the $\K$-vector space of all circuit divisors of $L(w|A)$ by $\calC (w|A)$, and we write $0 = 0C$. Two circuit divisors $D_1,D_2$ are called \emph{linearly equivalent} if their circuit vectors are the same, in which case we write $D_1 \sim D_2$.
\end{Def}

The purpose of this notation is to put an emphasis on the algorithmic process of combining circuits, over the pure consideration of the circuit vector. Indeed, in general, the same circuit vector can be obtained from different formal linear combinations of circuits.

The principal objects for solving the sparse linear system are the respective spans of particular and general regression circuits, which we will term regression space and general circuit space. They can be seen as analogues to the particular and general solutions occurring in the theory of differential equations: for solving the linear system accurately and efficiently, we need to find one particular regression circuit in the regression space, and a sufficient number of general circuits in the general circuit space.
\begin{Def}
Let $L(w|A)$ be a regression matroid.  The \emph{regression space}, is
affine span
\[
	\calC_p (w|A) := \aff\{C \in \calC(w|A) : \text{$C$ is a particular regression circuit}\}
\]
(where $\aff$ denotes the affine hull), and the \emph{general circuit space} is
\[
	\calC_c(w|A) := \spn\{ C \in \calC(w|A) : \text{$C$ is a general regression circuit}\}
\]
Elements of $\calC_p (w|A)$ are called \emph{particular regression divisors}, elements of $\calC_c (w|A)$ are called general regression divisors.
\end{Def}
Since $\calC_p(w|A)$ contains circuit divisors, and not just the circuit vectors, is has richer structure than the left kernel $A$.

The relationship between the two spaces is:
\begin{Lem}\label{Lem:circuit-spaces}
Let $L(w|A)$ be a regression matroid with $w\neq 0$. Then:

\begin{description}
\item[(i)] $\Ker A = \{\lambda_D\;: D\in \calC_c(w|A)\}$
\item[(ii)] Let $C_p$ be
a fixed particular regression circuit vector. Then\\
	$\calC_p(w|A) = \{C : C\sim C_p + D\;\mbox{with}\; D\in \calC_c(w|A)\}$
\end{description}
\end{Lem}
\begin{proof}
The first equality follows from the fact that the kernel vectors with minimal support span the kernel. For the second, it suffices to prove that for two particular regression circuits $C_1,C_2$, there is $D\in \calC_c(w|A)$ such that $C_1- C_2\sim  D$.
By definition, $(\lambda_{C_1} - \lambda_{C_2})A = 0$, therefore $\lambda_{C_1-C_2}\in \Ker A$ the statement then follows from the first equality.
\end{proof}

One could interject that representing the left kernel vectors of $A$ in terms of circuits and circuit divisors is unnecessarily complicated. The theoretical estimator will be formulated, in terms of the non-zero entries of the circuit vector $\lambda_C$; also, the algorithmic procedure will also benefit from treating the circuits as sets of indices instead of the circuit vector $\lambda_C$. The reader is invited to think about circuits and divisors simultaneously in terms of the circuit vectors plus the information which entries are non-zero, but we think that the notion of circuit divisors makes more clear where the advantages of our algebraic combinatorial method lie.

The final object we need to define before describing the estimation procedure is linear spans of divisors:
\begin{Def}
\begin{description}
\item[(i)] A $\K$-vector space $\calL_c\subseteq \calC_c (w|A)$ closed under linear equivalence is called \emph{linear system of general circuit divisors}, or short, \emph{general system} of circuits.
\item[(ii)] A $\K$-affine space $\calL_p\subseteq \calC_p (w|A)$ closed under linear equivalence is called \emph{affine system of particular circuit divisors}, or short, \emph{particular system} of circuits.
\end{description}
A general/particular system of circuit divisors $\calL$ is said to be generated by circuit divisors $C_1,\dots, C_m$ if every element in $\calL$ is linearly equivalent to a linear/affine combination of the $C_i$. In this case, the $C_1,\dots, C_m$ are called \emph{generating system} of $\calL$, and if $m$ is additionally minimal, they are called a \emph{basis} of $\calL$ (in both the linear/affine cases).
\end{Def}

\begin{Rem}\label{Rem:genpart}
An important example of particular systems is given as follows: let $\calL_c$ be a general system, and $C_p\in \calC_p(w|A)$, for example $C_p$ a particular regression circuit. Then, $C_p + \calL_c:=\{C_p+C_c\;:\;C_c\in\calL_c\}$ is a particular system.
\end{Rem}

Particular systems will be one of the main ingredients in estimating the projection $\langle w,x\rangle$. The above remark shows that to this end, it suffices to acquire a single particular circuit and some general system of circuits.

\subsubsection*{Examples of regression matroids}
We give some examples of regression matroids and discuss the structure of their
regression circuits.
\begin{Ex}[name={Uniform regression matroid}]\label{Ex:uniform-rm}
If $A$ is generic, then $L(w|A)$ is a quotient of a rank $n$ uniform matroid, so all special regression
\end{Ex}

\begin{Ex}[name={Generic low-rank}]\label{Ex:low-rank-rm}
If $A$ is generic of rank $r$, any $r$ rows will form a special regression circuit, any $r+1$ rows a general one.  In particular, if $r\ll n$, then  special regression circuits have \emph{sparse} support.
only rank $r$.  As in Example \ref{Ex:uniform-rm}, the special regression circuits are easy to
find, but now they are \emph{sparse}, provided $r\ll n$: only $r$ rows are required.
\end{Ex}

\begin{Ex}[name={Graphic regression matroids}]\label{Ex:graphic}
as a basis for the space of cycles. The regression space is therefore equivalent to the first homology of the graph $G$, a basis of which can be efficiently computed in $O(n + N)$ time.
\end{Ex}

The potentials Example \ref{Ex:potentials} and matrix completion Example \ref{Ex:matrix-completion}
both give rise to graphic regression matroids.  We will explore strategies for finding good
sets of special regression circuits in this case below.

Another combinatorial example comes from matrices with \emph{sparse filling patterns}
and \emph{generic} non-zero entries.
\begin{Ex}[name={$(1,\ell)$}-sparsity matroids]
We now define a kernel which will be key in our estimation procedure.  The intuition
behind the technical definition is that special regression circuit vectors define linear
If the rows of $A$ have at most $d$ generic non-zero entries, and the kernel of $A$
is spanned by $\ell\le d - 1$ generic vectors, then $L(w|A)$ is the
quotient of a $(1,\ell)$-sparsity matroid on a $d$-hypergraph.
\end{Ex}

\section{Matroid Regression}\label{sec:estimator}
The strategy for constructing the matroid regression estimator is as follows: each special regression divisor produces one exact estimate for the evaluation $\langle w,x\rangle$. The estimator is obtained for the choice of regression divisor minimizing variance. Since the special regression divisors form an affine space, on which variance is a quadratic form, we obtain the variance minimizing estimate as explicit solution to a quadratic system. The major algorithmical advantage of the matroid regression view was outlined in Remark~\ref{Rem:genpart}: after finding one special regression circuit, general regression circuits that are easier to find can be used to produce more special circuits in order to decrease the variance and this the estimation error.

\subsection{An Unbiased Estimator} \label{sec:est.estimcomp}
Recall problem (P): we are provided with the data for a regression matroid $L(w|A)$, with $A\in\K^{N\times n}$ known and $x\in \K^n$ unknown,
and want to estimate an evaluation $\iprod{w}{x}$ of the unknown signal $x$, from $b = Ax + \varepsilon$.
Let $\Sigma$ be the covariance matrix of the $N$-dimensional random vector $\varepsilon$.
First we construct circuit-vector estimators.
\begin{Prop}\label{Prop:onecirc}
Let $D$ be a particular regression divisor of $L(w|A)$, with circuit vector $\lambda_D$. Then,
$$\widehat{\gamma} (D):=\left\langle \lambda_D,b\right\rangle$$
is an unbiased estimator for $\langle w,x\rangle$ with variance
$\Var (\widehat{\gamma} (D)) = \lambda_D^*\Sigma \lambda_D.$
Conversely, all unbiased estimators linear in $b$ are of the type $\widehat{\gamma}(D)$ for some particular regression divisor $D$.
\end{Prop}
\begin{proof} By linearity of expectation and centeredness of $\varepsilon$, it follows that
\begin{align*}
\E(\widehat{\gamma}(D))=\left\langle \lambda_D, \E (b)\right\rangle = \lambda_D A x = \langle w,x\rangle,
\end{align*}
where the last equality follows from the fact that $\lambda_D$ is a circuit divisor - thus $\widehat{\gamma} (D)$ is unbiased. The statement for the variance follows from bilinearity of covariance via the equation
$$\Var (\widehat{\gamma}(D)) = \Var (\langle \lambda_D,b\rangle ) = \lambda_D^* \Var(b) \lambda_D = \lambda_D^* \Var(\varepsilon) \lambda_D.$$
The converse statement follows from Lemma~\ref{Lem:circuit-spaces}.
\end{proof}

Proposition~\ref{Prop:onecirc} shows how a good estimator for $\langle w,x\rangle$ can be obtained: find a divisor $D$ with small variance. Since the latter is quadratic in $\lambda_D$, this can be reduced to a quadratic optimization problem. However, there is one major issue with the present formulation: such an optimization would be essentially over $\lambda_D$, not in terms of the circuits, and eventually involve inversion of an $(N\times N)$ matrix - therefore nothing is gained yet with respect to the pseudo-inversion done in usual linear regression. To address this issue, we will express the variance from Proposition~\ref{Prop:onecirc} in terms of circuits and divisors.

\subsection{The Circuit Kernel}\label{Sec:circuit-kernel}
The circuit kernel is the analogue of the covariance matrix of the estimator $\widehat{\gamma}$, but represented in the coordinates induced by circuits and formal divisors. It yields a quadratic form on the circuit space $\calC(w|A)$, allowing optimization to take place over the combinatrial structure of the circuits as compared to the circuit vectors $\lambda_*$.

\begin{Def}
Fix a covariance matrix $\Sigma\in \K^{N\times N}$. For two regression divisors $D_1,D_2$ with circuit vectors $\lambda_1,\lambda_2$, we define the \emph{circuit kernel function}
$$k(D_1,D_2)= \lambda_1^* \Sigma \lambda_2.$$
For a collection $D_1,\dots, D_m$ of regression divisors and $\Sigma$,
we define the \emph{circuit kernel matrix} $K$ to be the $(m\times m)$ matrix
which has $k(D_i,D_j)$ as entries.
\end{Def}

\begin{Lem}\label{Lem:circpsd}
The circuit kernel is a positive semi-definite bilinear form on $\calC(w|A)$.
For a particular regression divisor $D\in\calC_P(w|A),$ it holds that $\Var (\widehat{\gamma} (D)) = k(D,D)$.
\end{Lem}
\begin{proof}
The matrix $\Sigma$ is positive semi-definite as covariance matrix of a random variable.
Therefore, there is a Cholesky decomposition $\Sigma = U^\top U$ with $U\in \R^{N\times N}$.
Observe that by definition, any circuit kernel matrix $K$ will be of the form
$K=\Lambda^*\Sigma\Lambda$ for $\Lambda\in\K^{N\times m}$ and some $m$.
Therefore, $K=(U \Lambda)^* (U\Lambda)$ is a positive semi-definite matrix, which implies positive semi-definiteness of $k$. The second statement follows from Proposition~\ref{Prop:onecirc} and the definition of $k$.
\end{proof}

\begin{Prop}\label{Prop:Varest}
Let $\calL$ be a particular system, generated by $C_1, \ldots, C_m$. The quadratic form $k(D,D)$ is minimized for $D\in\calL$ by exactly the divisors
$$
  D=\sum_{i=1}^n \alpha_i C_i, \;\mbox{where}\;\alpha \in \left(K^{-1}\mathbf {1}\right)\left(\mathbf {1}^\top K^{-1}\mathbf {1}\right)^{-1},
$$
$\mathbf {1}$ is the vector of ones, $K$ is the $(m\times m)$ kernel matrix with entries $k(C_i,C_j)$, and $K^{-1} \mathbf {1} =\{x\in\K^n\;:\;Kx = \mathbf {1} \}$.
\end{Prop}
\begin{proof}
Since $\calL$ is a particular system, it holds that $D\in\calL$ if and only if $\mathbf {1}^\top \alpha = 1$. Bilinearity of $K$ implies that $k(D,D)=\alpha^\top K \alpha$. From this, we obtain the Lagrangian
$$L(\alpha,\xi)= \alpha^\top K \alpha +  \xi \left(1- \mathbf{1}^\top\alpha\right),$$
where the slack term models the condition $\mathbf{1}^\top\alpha=1$. A straightforward computation yields
\begin{align*}
\frac{\partial L}{\partial\alpha}& = 2 K \alpha - \xi \mathbf {1}
\end{align*}
By Lemma~\ref{Lem:circpsd} $K$ is positive semi-definite, therefore
$\alpha^\top K \alpha$ is convex, so the minimizers of $k(D,D)$ satisfying $\mathbf {1}^\top\alpha =1$ will exactly correpond to the
$\alpha \in K^{-1}\mathbf {1}/\mathbf {1}^\top K\mathbf {1}$.
\end{proof}

\subsection{The Optimal Estimator} \label{sec:est.estim}

We are now ready to give the final form of our estimator:

\begin{Thm}\label{Thm:est}
Let $\calL$ be a particular system. Let $D\in\calL$ be any divisor minimizing $k(D,D)$, as in Proposition~\ref{Prop:Varest}. Consider the estimator
$$\widehat{\gamma} (\calL) := \langle \lambda_D, b\rangle.$$
\begin{description}
\item[(i)] $\widehat{\gamma} (\calL)$ is independent of the choice of the minimizer $D$ of $k(D,D)$.
\item[(ii)] $\widehat{\gamma} (\calL)$ is an unbiased estimator for $\langle w,x\rangle$.
\item[(iii)] $\Var\widehat{\gamma} (\calL)= k(D,D) = \min_{D\in\calL} \Var \left(\widehat{\gamma}(D)\right).$
\end{description}
\end{Thm}
\begin{proof}
(i) follows from elementary linear algebra. (ii) follows immediately from Proposition~\ref{Prop:onecirc}. (iii) follows from Lemma~\ref{Lem:circpsd}.
\end{proof}

Theorem~\ref{Thm:est} indicates an algorithmic way to obtain good estimates for $\langle w,x\rangle$: namely, first find generators $C_1,\dots, C_m$ for a particular system; then determine the minimizer $D$ as described in Proposition~\ref{Prop:Varest}, keeping track of the circuit vector. Finally, compute $\widehat{\gamma}$. At the same time, Theorem~\ref{Thm:est} highlights several important advantages of our estimator $\calL$. First, computation of $\widehat{\gamma}$ involves only (pseudo-)inversion of an $(m\times m)$-matrix, as opposed to (pseudo-)inversion of an $(n\times n)$-matrix for the naive strategy - this is an advantage in the sparse setting, as we will show that $m$ can be chosen small in some common scenarios. Second, Theorem~\ref{Thm:est}~(i) in particular shows that the estimate does not depend on the particular generating system chosen for $\calL$; therefore, following Remark~\ref{Rem:genpart}, we may choose a system of the form $D_i = C_p + C_i$, where $C_i$ are general circuits, and $C_p$ is the same particular circuit for all $D_i$. This means, for each new $w$, we only need to find a single particular circuit, while the system of general circuits given by the $C_i$ needs only to be changed when $A$ changes. Due to the bilinear equality $k(D_i,D_j)=k(C_p,C_p)+k(C_p,C_i)+ k(C_p,C_j)+k(C_i,C_j)$ this also means that the kernel matrix $K$ has to be computed only once per $A$.

\subsection{Algorithms}\label{sec:estimator.algos}
We provide algorithms computing the estimated evaluation and variance bounds for the error.

Since the the circuit vectors, the circuit kernel matrix $K$, and the optimal $\alpha$ are required for both,
we first compute those, given a collection of circuits $C_1,\dots, C_m$. Algorithm~\ref{Alg:alpha} outlines informal steps for this. We use MATLAB notation for submatrices and concatenation.
\begin{algorithm}[h]
\caption[Computes circuit kernel $K$ and $\alpha$]{Computes circuit kernel $K$ and $\alpha$.\\
\textit{Input:} $A,w,$ circuits $C_1,\dots, C_m$, covariance matrix $\Sigma$.\\
\textit{Output:} circuit union $C$, circuit vector matrix $\Lambda$, kernel matrix $K$ and minimizer $\alpha$. \label{Alg:alpha}}
\begin{algorithmic}[1]
    \item[1:] For all $i=1\dots m$, compute the circuit vector $\lambda_i$ of the $C_i$ as the normalized left kernel vector of the matrix $[A[C_i,:];w]$ \label{Alg:alpha.step1}
    \item[2:] Write the $\lambda_i$ as rows of a matrix $\Lambda$, with rows indexed by $C$.
    \item[3:] Write $C=C_1\cup\dots\cup C_m$.
    \item[3:] Compute the kernel matrix $K = \Lambda^* \cdot\Sigma[C,C]\cdot \Lambda$ \label{Alg:alpha.step3}
    \item[4:] Calculate $\alpha = \left(K^{-1}\mathbf {1}\right)\left(\mathbf {1}^\top K^{-1}\mathbf {1}\right)^{-1}.$
    \item[5:] Output $\Lambda$, $K$ and $\alpha$.
\end{algorithmic}
\end{algorithm}
The computations in steps 1 %~\ref{Alg:alpha.step1}
and 3 %~\ref{Alg:alpha.step3}
can be done fairly efficiently, since while $\Sigma$ or $A$ may be huge, the circuits $C_i$
select only small submatrices. In an optimal scenario, increasing the number of rows of $A$ has only a
small or negligible effect on the size of $C$. Note that inputting $b$ is not required,
therefore Algorithm~\ref{Alg:alpha} needs not to be rerun if $A,w,\Sigma$ stay the same, i.e.,
if it is only the signal $x$ which changes.

Algorithms~\ref{Alg:Xalpha} and~\ref{Alg:Var} takes the computed invariants from Algorithm~\ref{Alg:alpha} and computes estimates for $\langle w,x\rangle$ and its variance. These algorithms consist only of multiplications, and again can be made efficient by the fact that the occurring matrices are of size at most $\# C$, therefore again controlled by the choice of circuits.
\begin{algorithm}[h]
\caption[Estimates the evaluation $\langle w,x\rangle$.]{\label{Alg:Xalpha} Estimates the evaluation $\langle w,x\rangle$.\\
\textit{Input:} $w,A,b$, a collection of circuits $C_1,\dots, C_m$, covariances $\Sigma$.\\
\textit{Output:} The variance-minimizing estimate $\widehat{\gamma}(\alpha)$ for $\langle w,x\rangle$. }
\begin{algorithmic}[1]
    \item[1:] Compute $\Lambda $ and $\alpha$ with Algorithm~\ref{Alg:alpha}.
    \item[2:] Write $C=C_1\cup \dots \cup C_m$ (this could also be obtained from Algorithm~\ref{Alg:alpha})
    \item[3:] Return $\alpha^\top\cdot \Lambda \cdot b[C]$ as an estimate.
\end{algorithmic}
\end{algorithm}

\begin{algorithm}[h]
\caption[Estimates the variance of the evaluation $\langle w,x\rangle$.]{\label{Alg:Var} Estimates the variance of the evaluation $\langle w,x\rangle$.\\
\textit{Input:} $A,w$, a collection of circuits $C_1,\dots, C_m$, covariances $\Sigma$.\\
\textit{Output:} The variance lower bound for $\log (A_{ij})$. }
\begin{algorithmic}[1]
    \item[1:] Calculate $K$ and $\alpha$ with Algorithm~\ref{Alg:alpha}.
    \item[2:] Return $\alpha^\top\cdot K \cdot \alpha$.
\end{algorithmic}
\end{algorithm}

Algorithm~\ref{Alg:Var} can be used to obtain the variance bound independently of the observations in $b$ - therefore an error estimate which is independent of the algorithm which does the actual estimation.\\

We would further like to note that the size of all matrices multiplied or inverted in the course of all three algorithms is bounded by the cardinality of the circuit union $C$. The only matrix of potentially larger size is $[A[C_i,:];w]$ which can have more columns than $\#C$, up to $n$. However, there is no noise on $A$, and this matrix is used only to compute the unique (up to multiplicative constant) left kernel vector, so it can be replaced by the matrix consisting of any $\#C +1 $ linearly independent columns. Therefore, once a suitable circuit basis $C_1,\dots, C_m$ is known which is accurate enough, $\langle w,x\rangle$ can be estimated in complexity depending only on $\#C$, and not on $N$ or $n$.

\subsection{On Finding Circuits}\label{sec:estimator.findcirc}

While the algorithms presented in section~\ref{sec:estimator.algos} are fairly fast and near-optimal by the considerations in sections~\ref{Sec:circuit-estimator-theory} and~\ref{sec:estimator.optimality}, they highly rely on the collection of circuits which is input and which determines the submatrix to consider. Therefore, one is tempted to believe that the difficult problem of inverting $A$ has merely been reduced to a combinatorial problem which is more difficult. The point here is again that if $A$ and $w$ are sparse, or if there is different combinatorial structure implying small circuits, this combinatorial problem has a comparably simple solution in practical settings. For example, if not much is known about $A$, but it has small circuits that are well-dispersed, one can attempt to find circuits via $\ell^1$-minimization, e.g., by solving the convex program
$$\min \|\lambda\|\quad \mbox{subject to}\quad \lambda^\top\cdot A[D] = 0,$$
where $D$ is is a randomly chosen subset of a number of columns which is likely to contain a circuit. On the other hand, if the rows of $A$ and/or $w$ have a specific combinatorial structure, for example related to properties of graphs, this can open up the problem to efficient algorithms which scale with the problem's sparsity instead of its size. One can regard the matrix completion algorithm from~\cite{KirThe13Rank1} as a proof of concept for this, since the computation of the graph homology may be done in a local neighborhood around the missing entry whose size is constant, we will explain this in more detail in section~\ref{sec:expl}. We will also list more examples with different combinatorial features that can be treated in this way.

\subsection{Example Cases}\label{sec:expl}
The algorithms outlined in section~\ref{sec:theory} provide a fast and stable way of computing the evaluation once enough circuits have been identified. One main advantage of our strategy is that for each matrix $A$, the circuits need to be computed only once, and can be applied for different signals $x$. Furthermore, if the sparse matrix $A$ (or its dual $A'$) is highly structured - as it frequently occurs when analyzing network structure - then so are the circuits, in which case they can be obtained by combinatorial algebraic methods. We list some basic examples for demonstration purposes.

\subsubsection*{The Sample Mean and Linear Regression}
Both sample mean and ordinary least squares regression can be recovered as special cases of matroid regression. The sample mean is obtained for setting $A$ to be an $N$-vector of ones, $w=1$ and $\Sigma$ the identity matrix - regression circuits consist of exactly one element, with the circuit vector being the corresponding standard basis vector. Least squares regression is obtained for setting $\Sigma$ to be the identity and estimating evaluations for $w= e_i, 1\le i\le n$ with $e_i$ being an orthonormal system for $\K^n$.

\subsubsection*{Multiple Observations}
A behavior related to sample mean can be observed if multiple copies of the same row occur in $A$. In this case, a regression circuit will contain exactly one of those, and there will be a special regression circuit of the same type for each of the copies. Furthermore, for each pair of copies, a general circuit will appear containing exactly that pair. In order to prevent multiplicative growth of the number of circuits, it is suggested to pool multiple observations in a single one by taking the covariance-weighted mean.

\subsubsection*{Denoising}
A related case is if $A$ contains $w$ as a row. Here, that row will occur as a special regression circuit with only one element. Applying matroid regression in this case will trade off the noise in that single observation through the relations with other rows, therefore can be interpreted as a denoising of that observation. Even $w$ occurs multiple times as a row of $A$, matroid regression will in general improve over merely taking the covariance-weighted sample mean of those rows' observations.

\subsubsection*{Measuring Potentials}

We consider the case where $x$ corresponds to a potential, and differences are measured. In this case the rows of $A$ take the form $e_i-e_j$ with $e_i$ the standard basis for $\K^n$; assume that $w = e_k-e_\ell$ is of the same form. Let $G$ be the oriented graph with $n$ nodes which has an edge $(i,j)$ if and only if $A$ has a row $e_i-e_j$. Then the following characterization for regression circuits and general circuits can be shown: a set of edges is a special regression circuit if and only if it forms a path from $k$ to $\ell$ contained in $G$ - including possibly the edge $(k,\ell)$ itself in case $w$ occurs as a row of $A$. The corresponding circuit vector consists of ones. A set of edges is a general circuit if and only if it is a cycle contained in $G$. Small circuits can therefore be efficiently found by finding elements in the first graph homology of $G$ around the edge $(k,\ell)$.

\subsubsection*{Sparse Sums}
The case where the rows of $A$ are of form $e_i+e_j$, and $w= e_k+e_\ell$ is very similar. Let $G$ be the (simple) graph with $n$ vertices and the same edge assignment as above. In this case, the special regression circuits will be exactly paths of odd length from $k$ to $\ell$ contained in $G$, with circuit vectors being alternatingly $-1$ and $1$, starting with $-1$. General circuits will be cycles of even length, with circuit vectors alternatingly $1$ and $-1$. As in the potentials case, a search of the first graph homology will provide cycles near $(k,\ell)$ efficiently.

\subsubsection*{Low-Rank Matrix Completion}
By taking logarithms, compare the general strategy in~\cite{KirThe13Rank1}, the rank one matrix completion problem can be transformed to the following linear problem: write the true rank $1$ matrix $X\in \R^{m\times n}$ ($X$ = the $A$ from the cited paper) as $X=uv^\top$ with $u\in \R^m$ and $v\in \R^n$. Then, $x$ is an $(m+n)$-vector that is concatenation of component-wise $\log u$ and $\log v$. The rows of the matrix $A$ consist of concatenations $(e_i, e'_j)$ of standard basis vectors $e_i\in \R^m$ and $e'_j\in \R^n$, being present if the entry $(i,j)$ is observed; $w$ is of the same form, corresponding to the unobserved entry $(k,\ell)$. This exposes rank one matrix completion as a sub-case of the ``sparse sums'' scenario discussed above. Note that the graph $G$ is always bipartite due to how $A$ was constructed, and that the missing entry of the matrix can be completed from a local neighborhood of entries by the same principles applying to the search of the graph homology.

With this reduction, \cite[Theorem 3.10]{KirThe13Rank1} is directly implied by Theorem~\ref{Thm:opt} from
Section~\ref{sec:estimator.optimality} below.

Furthermore, the theory for matrices of arbitrary rank outlined in~\cite{KTTU12} can be interpreted as a non-linear generalization; furthermore, it indicates that matroid regression is also a viable tool for solving systems of equations carrying a structure of non-linear matroid.

\subsubsection*{Measuring Matrices and Phase Recognition}
As the low-rank matrix completion scenario indicates, the linear techniques can also be used if the signal $x$ is in reality a matrix $X$, and each \emph{row}  of $A$ is the vectorization a matrix $Z_i$ of the same format, for example $Z_i = u_iv_i^*$, in which case $b_i = \Tr (X Z_i) + \varepsilon_i = v_i^* X u_i + \varepsilon$. Phase recognition is a special case of this example where $u_i=v_i$ for every $i$, and $X$- is a Hermitian rank one matrix.
If one of $u_i,v_i$ is always a standard basis vector $e_i$, and the other is $e_k\pm e_\ell$, this is a special subcase of the potentials or sparse sums scenario. If both are of the form $e_k - e_\ell$, the circuits correspond to the first syzygies of the rank one determinantal variety. In general, there is no easy way in which the circuits of the $Z_i$ relate to those of $u_i$ and $v_i$, but this is an interesting question to ask, in particular for the highly regular measurement designs employed in phase recognition.

The case where $X$ is a symmetric, non-symmetric or partially symmetric tensor of higher degree, and where $Z_i$ are outer products, can be seen as a generalization.

\subsubsection*{Low-Rank $A$}
In case the matrix $A$ is of rank $r$ and otherwise non-degenerate, matroid regression suggests an algorithm for inversion of $A$ which includes inversions of $(r\times n)$ matrices only. Namely, any $r$ rows will form a general regression circuit; if $A$ is split into $N/r$ disjoint row-blocks $A_i$ of size $r$, then the estimator in Proposition~\ref{Prop:Varest} will be a weighted sum of estimates of the form $w^\top \A_i^{-1}b$, which is of lower complexity than inversion of $A$ since matrix inversion scales with an exponent at least $2$ in the size. This can be seen as an arbitrary rank generalization of the sample mean scenario, where the rank is $1$.

\section{Properties of the Matroid Regression Estimator}\label{Sec:properties}

This section shows some key properties of the matroid regression estimator. Summarizing, we how that $\widehat{\gamma} (\calL)$ is the best linear unbiased estimator (BLUE) for any noise model, and the minimum variance unbiased estimator (MVUE) as well if the noise is Gaussian (homo- or heteroscedastic) - among all estimators that use only information in rows related to the particular system $\calL$. Furthermore, we show a monotonicity result, showing that the variance of the estimator $\widehat{\gamma} (\calL)$ drops as $\calL$ is enlarged. These results do not only show that the estimator $\widehat{\gamma} (\calL)$ is optimal, but also that Algorithm~\ref{Alg:Var} computes a tight lower bound on the estimation error without actually estimating the evaluation $\langle w,x\rangle$, therefore provides a lower error bound for any method that is employed.

\subsection{Monotonicity and Complexity-Accuracy-Tradeoff}
The first important property of the estimator $\widehat{\gamma}(\calL)$ is being monotone with respect
to inclusion of $\calL$; that is, adding more circuits will only improve the estimator:

\begin{Thm}\label{Thm:monotone}
Let $\calL,\calL'$ be particular systems with $\calL\subseteq \calL'$. Then, $\Var\left(\widehat{\gamma}(\calL')\right)\le \Var\left(\widehat{\gamma}(\calL)\right).$
\end{Thm}
\begin{proof}
This follows from Theorem~\ref{Thm:est}~(iii).
\end{proof}

Theorem~\ref{Thm:monotone} can be interpreted as a complexity-accuracy-tradeoff incurred by the amount of locality. More specifically, making the particular system $\calL$ smaller will make Algorithm~\ref{Alg:alpha} run faster, but leads to an increased expected error in the estimate. Conversely, adding circuits and this enlarging $\calL$ will make the estimate more accurate, but the algorithmic computation more expensive.\\

\subsection{Optimality Amongst Linear Unbiased Estimators}\label{Sec:circuit-estimator-theory}
The estimator $\widehat{\gamma}(\calL)$ has already been shown to be variance minimizing for choice of $D$ in the particular system $\calL$ in Theorem~\ref{Thm:est}~(iii); we will make a similar statement relating it to using different entries of the vector $b$.

\begin{Def}
Let $\calL$ be a particular/general system of divisors. The \emph{support} of $\calL$ is the inclusion-wise maximal set $I\subseteq [N]$ such that $\calL$ contains all particular/general circuits contained in $I$. Conversely, for $I\subseteq [N]$, denote $\calL (I) := \aff \{C\;:\; C\in\calC_p(w|A), C\subseteq I\}$ or, equivalently, $\calL (I) := C_p + \{C\;:\;C\in \calC_c (w|A), C\subseteq I\}$ for some particular circuit $C_p\subseteq I$.
\end{Def}

\begin{Thm}\label{Thm:linear}
Let $I\subseteq [N]$, and let $\widehat{\gamma}'$ be any unbiased estimator for $\langle w,x\rangle$ linear in the entries $b_i,i\in I$, with coefficients depending only on $A$ and $w$. Let $\calL = \calL(I)$. Then, $\Var\left(\widehat{\gamma}(\calL)\right)\le \Var\left(\widehat{\gamma}'\right)$.
\end{Thm}
\begin{proof}
This is implied by Proposition~\ref{Prop:onecirc} which states that any estimator, linear in $b_i$ and unbiased, is of the form $\widehat{\gamma}(D)$ for some divisor $D\in \calL (I)$. The statement is then implied by Theorem~\ref{Thm:monotone}.
\end{proof}

Theorem~\ref{Thm:linear}, together with the characterization of estimators linear in $b$ in Proposition~\ref{Prop:onecirc}, implies that $\widehat{\gamma}$ is the best linear unbiased estimator (BLUE) for $\langle w,x\rangle$, and is an analogue of the Gauss-Markov theorem in our case.

\subsection{Universal Variance Minimization}\label{sec:estimator.optimality}
In this section, we will show that our estimator $\widehat{\gamma}$ is optimal in two further ways:
first, for Gaussian noise, $\widehat{\gamma}$ is a sufficient and complete statistic, and thus
the minimum variance unbiased estimator (MVUE); second for general centered noise, $\widehat{\gamma}$ has minimum variance
among unbiased estimators independent of $x$ and $\varepsilon$. This ``noise optimality''
implies that lower-variance estimators need to use additional information about either
the unknown signal $x$ or the distribution of the noise.

We first prove optimality for Gaussian noise:

\begin{Thm}\label{Thm:Gaussian}
Let $\calL$ be a particular system. Assume that the noise $\varepsilon$ is multivariate Gaussian. Then, the estimator $\widehat{\gamma}(\calL)$ is
\begin{itemize}
\item[(i)] a complete statistic with respect to the parameter $\langle w, x\rangle$ and observations $b_i,i\in \supp \calL.$
\item[(ii)] a sufficient statistic with respect to the parameter $\langle w, x\rangle$ and observations $b_i,i\in \supp \calL.$
\item[(iii)] the minimum variance unbiased estimator for the parameter $\langle w, x\rangle$ and observations $b_i,i\in \supp \calL.$
\end{itemize}
\end{Thm}
\begin{proof}
(i) and (ii): Without loss of generality, we can assume that $\supp\calL$ is all rows, otherwise, we remove the rows from $A$ not contained in $\supp\calL$.

Let $D_1,\dots, D_k$ be a basis for $\calL$, and let $\Lambda$ be the $(k\times N)$ matrix whose columns are the $\lambda_{D_i}$. Then, by definition, it holds that
$$\langle \Lambda, b\rangle = \langle w, x\rangle\cdot \mathbf {1} + \langle \Lambda, \varepsilon\rangle.$$
If we know that $\langle \Lambda, b\rangle$ is a complete and sufficient statistic for $\langle w, x\rangle$, we are done by virtue of the following argument: having reduced the original problem to linear regression, we observe that the BLUE of the latter is exactly $\widehat{\gamma}(D)$ with $D$ minimizing $k(D,D)$, which is also known to be a complete and sufficient statistic and thus the MVUE, see e.g.~\cite[section 8.3, example 8.3]{Cox1974} (for homoscedastic noise, the case of heteroscedastic noise follows from applying an appropriate linear transform). By Theorem~\ref{Thm:est}, the estimator $\widehat{\gamma}(D)$ is the same as $\widehat{\gamma}(\calL)$, proving the statement.\\

We now prove the remaining {\bf claim:} $\langle \Lambda, b\rangle$ is a complete and sufficient statistic for $\langle w, x\rangle$. To prove this, we observe that we can write $x$ as an orthogonal decomposition $x = x_w + x_w^\perp$ where $x_w= \frac{x \langle w, x\rangle}{\langle w, x\rangle^2}$ is the orthogonal projection of $x$ on $\spn w$, thus $\langle w, x\rangle = \langle w, x_w\rangle$ and $\Lambda A x = \Lambda A x_w$. Since the columns of $\Lambda$ are a basis of $\calL$, the matrix $\Lambda A$ acts, by definition, bijectively on $x_w$, which proves claim 1.\\

(iii) follows from (i) and (ii) via the Lehmann–Scheffé-Theorem.
\end{proof}

It is straightforward to extend the proof to other suitable members of the exponential family. By the Pitman-Koopman-Darmois theorem, it is unreasonable though to expect sufficiency for non-exponential distributions. However, we can prove a similar conclusion which relies on replacing sufficiency by universality with respect to the underlying signal:

\begin{Thm}\label{Thm:opt}
Let $\calL$ be a particular system. Let $\widehat{\gamma}'$ be an estimator for $\langle w, x\rangle$ which is unbiased for all choices of $x$, of the form
$\widehat{\gamma}' = f (b,\;b\in\supp\calL)$ with some $f\in L^2 (\K^n)$. Then, $\Var\left(\widehat{\gamma}(\calL)\right)\ge \Var \left(\widehat{\gamma}'\right)$ for any choice of noise $\varepsilon$.
\end{Thm}

The proof of Theorem will be split in several statements. It will be immediately implied by Theorem~\ref{Thm:mainthm_linest} below, which states that an universally unbiased estimator always has the form $\widehat{\gamma} (\alpha)$, and Proposition~\ref{Prop:Varest}.

\begin{Lem}
\label{Lem:probdense}
Let $f\in L^2 (\K^n)$. If $\mathbb{E}(f(X)) = 0 $ for all random variables $X\in\K^n$ for which this expectation is finite, then $f = 0$.
\end{Lem}
\begin{proof}
By definition, the statement is equivalent to: Let $f\in L^2 (\K^n)$ be a smooth function. If
$$\langle f,p\rangle = \int_{\K^n} f(x) p(x)\; dx = 0$$
for all smooth functions $p:\K^n\rightarrow \K$ that fulfill $\langle p,1\rangle = 1$ and $p(x)\ge 0$ for all $x\in\K^n$, then $f=0$.

Now if $\langle f,p\rangle = 0$ for all smooth functions $p:\K^n\rightarrow \R$ that fulfill $\int_{\K^n} p(x)\; dx = 1$ and $p(x)\ge 0$, then $\langle f,g\rangle = 0$ for all functions $g\in L^2(\K^n)$, since the span of all such square-integrable $p$ (which includes the simple functions) is dense in $L^2(\K^n)$. This implies $\langle f,f\rangle = \|f\|^2 = 0$, therefore $f = 0$.
\end{proof}

\begin{Lem}
\label{Lem:centered}
Let $f\in L^2 (\K^n)$. If $\mathbb{E}(f(X)) = 0 $ for all centered random variables $X\in\K^n$ for which this expectation is finite, then $f$ is linear in $X$, i.e., of the form $f: x\mapsto  \langle \lambda, x\rangle$ with $\lambda \in \K^n$.
\end{Lem}
\begin{proof}
Denote $\phi_i: \K^n\rightarrow \K, x\mapsto x_i,$ and denote
$$G=\{g\in L^2 (\K^n)\;:\; \langle \phi_i, g\rangle <\infty \;\mbox{for all}\;1\le i\le n\}.$$
Note that $G$ is a $\K$-vector space. Further denote
$$H=G\cap \mbox{span}\{p\in L^2 (\K^n)\;:\; \langle 1, p\rangle = 1, \langle \phi_i,p\ge 0, p\;\mbox{smooth}\}.$$
Since the square-integrable probability distributions span $L^2 (\K^n)$, it follows that
$$H = \{g\in G\;:\;\langle \phi_i, g\rangle = 0\;\mbox{for all}\;1\le i\le n\}.$$
Therefore, $H$ is the kernel of the linear map
$$\varphi : G\rightarrow \K^n, g\mapsto (\langle \phi_1, g\rangle,\dots, \langle \phi_n, g\rangle).$$
This map is surjective, since e.g.~all Gaussians are in $G$.
By Lemma~\ref{Lem:probdense}, $G^\perp =\{0\}$, therefore $H^\perp$ is contained in a vector space isomorphic to $\K^n$. Since the $n$-dimensional $\K$-vector space $(\K^n)^\vee$ is contained in $H^\perp$, it follows that $H^\perp=(\K^n)^\vee$. From this, the statement follows.
\end{proof}

An application yields the following statement:
\begin{Prop}
\label{Prop:linearform}
Let $f\in L^2 (\K^n)$, let $\beta\in\K^n$. If $\E(f(X)) = \langle \beta, \E(X)\rangle $ for all $\K^n$-valued random variables $X$ for which the expectation $\E(f(X))$ is finite, then $f$ is of the form $f: x \mapsto \langle \beta , x  \rangle.$
\end{Prop}
\begin{proof}
Applying Lemma~\ref{Lem:centered} for the function $f-\langle \beta, . \rangle$ and the random variable $X - \E(X)$ yields that $f$ is of the form
$$f: x\mapsto \langle \beta, \E(X)\rangle + \langle \lambda,  x - \E (X)\rangle.$$
Since $f$ is only a function of $x$ and not of $\E(X)$ which can vary, the coefficient of $\E(X)$, which is equal to $\beta - \lambda$, must vanish, thus $\beta = \lambda$. Substituting yields the claim.
\end{proof}

\begin{Thm}
\label{Thm:mainthm_linest}
Let $A\in \K^{N\times n}$ and $x\in \K^n$. Let $\varepsilon$ be a centered and $\K^n$-valued random variable, let $b = Ax + \varepsilon$. For $c\in \spn A$, let $\widehat{\gamma}$ be an estimator for $\langle w,x\rangle$ of the form $\widehat{\gamma} = f(b)$ with $f\in L^2 (\K^n)$. If $\E (\widehat{\gamma}) = \langle w,x\rangle$ for all choices of $x$, and all choices of $\varepsilon$ for which this expectation is finite, then $\widehat{\gamma}$ is of the form $\widehat{\gamma} (D)$, as described in Proposition~\ref{Prop:onecirc}.
\end{Thm}
\begin{proof}
Since $w\in \spn A$, there exists $\beta\in \K^N$ such that $\beta^\top AX  = \langle w,x\rangle$. Linearity of expectation implies $\E (\langle \beta, b\rangle ) = \langle w,x\rangle$. Taking $X=b$ and this $\beta$ in Proposition~\ref{Prop:linearform} yields the claim.
\end{proof}

%\newpage
\section*{Acknowledgments}
LT is supported by the European Research Council under the European Union’s Seventh Framework Programme (FP7/2007-2013) / ERC grant agreement no 247029- SDModels.  This research was carried out at MFO, supported by FK's Oberwolfach Leibniz Fellowship.

\bibliographystyle{abbrvnat}
{\small
\bibliography{linear_matroidkernel}}

\begin{thebibliography}{11}
\providecommand{\natexlab}[1]{#1}
\providecommand{\url}[1]{\texttt{#1}}
\expandafter\ifx\csname urlstyle\endcsname\relax
  \providecommand{\doi}[1]{doi: #1}\else
  \providecommand{\doi}{doi: \begingroup \urlstyle{rm}\Url}\fi

\bibitem[Barrett et~al.(1994)Barrett, Berry, Chan, Demmel, Donato, Dongarra,
  Eijkhout, Pozo, Romine, and van~der Vorst]{Barrett1994}
R.~Barrett, M.~W. Berry, T.~F. Chan, J.~Demmel, J.~Donato, J.~Dongarra,
  V.~Eijkhout, R.~Pozo, C.~Romine, and H.~van~der Vorst.
\newblock \emph{Templates for the Solution of Linear Systems: Building Blocks
  for Iterative Methods}.
\newblock Society for Industrial and Applied Mathematics, 1994.

\bibitem[Chung(1997)]{Chung1997}
F.~R. Chung.
\newblock \emph{Spectral Graph Theory}.
\newblock Number Nr. 92 in CBMS Regional Conference Series. American
  Mathematical Society, 1997.

\bibitem[Cox and Hinkley(1974)]{Cox1974}
D.~R. Cox and D.~V. Hinkley.
\newblock \emph{Theoretical Statistics}.
\newblock Chapman \& Hall, 1st edition, 1974.

\bibitem[Davis(2006)]{Davis06}
T.~A. Davis.
\newblock \emph{Direct Methods for Sparse Linear Systems}, volume~2 of
  \emph{Fundamentals of Algorithms}.
\newblock SIAM, 2006.

\bibitem[Hackbusch(1993)]{Hackbusch1993}
W.~Hackbusch.
\newblock \emph{Iterative Solution of Large Sparse Systems of Equations}.
\newblock Springer, 1993.

\bibitem[Herman(1980)]{Herman1980}
G.~T. Herman.
\newblock \emph{Image Reconstruction from Projections: The Fundamentals of
  Computerized Tomography}.
\newblock Academic Press, 1980.

\bibitem[Kak et~al.(1988)Kak, Slaney, in~Medicine, and Society]{Kak1988}
A.~C. Kak, M.~Slaney, I.~E. in~Medicine, and B.~Society.
\newblock \emph{Principles of Computerized Tomographic Imaging}.
\newblock IEEE Engineering in Medicine and Biology Society, 1988.

\bibitem[Kir\'aly and Theran(2013)]{KirThe13Rank1}
F.~J. Kir\'aly and L.~Theran.
\newblock Error-minimizing estimates and universal entry-wise error bounds for
  low-rank matrix completion.
\newblock \emph{Advances in Neural Information Processing Science 2013}, 2013.

\bibitem[Kir\'aly et~al.(2012)Kir\'aly, Theran, Tomioka, and Uno]{KTTU12}
F.~J. Kir\'aly, L.~Theran, R.~Tomioka, and T.~Uno.
\newblock The algebraic combinatorial approach for low-rank matrix completion.
\newblock Preprint, arXiv:1211.4116v4, 2012.
\newblock URL \url{http://arxiv.org/abs/1211.4116}.

\bibitem[Saad(2003)]{Saad2003}
Y.~Saad.
\newblock \emph{Iterative Methods for Sparse Linear Systems: Second Edition}.
\newblock Society for Industrial and Applied Mathematics, 2003.

\bibitem[Tewarson(1973)]{Tewarson73}
R.~P. Tewarson.
\newblock \emph{Sparse Matrices}.
\newblock Academic Press, 1973.

\end{thebibliography}

\end{document}